\documentclass[reqno,11pt]{amsart}
\usepackage{amsmath, latexsym, amsfonts, amssymb, amsthm, amscd}
\usepackage{multirow}
\usepackage{hyperref}
\usepackage{appendix}
\usepackage{chngcntr}
\usepackage{etoolbox}
\usepackage{lipsum}
\usepackage{graphicx}
\usepackage{comment}
\usepackage{xcolor}

\usepackage[utf8]{inputenc} 

\usepackage{datetime}

\usepackage{cancel}
\usepackage{graphics,epsf,psfrag}
\numberwithin{equation}{section}

\newtheorem{theorem}{Theorem}[section]
\newtheorem{lemma}[theorem]{Lemma}

\newtheorem{cor}[theorem]{Corollary}
\newtheorem{rem}[theorem]{Remark}
\newtheorem{definition}[theorem]{Definition}

\newcommand{\R}{\mathbb{R}}

\renewcommand{\tilde}{\widetilde}

\DeclareMathSymbol{\leqslant}{\mathalpha}{AMSa}{"36} 
\DeclareMathSymbol{\geqslant}{\mathalpha}{AMSa}{"3E} 
\DeclareMathSymbol{\eset}{\mathalpha}{AMSb}{"3F}     
\renewcommand{\leq}{\;\leqslant\;}                   
\renewcommand{\geq}{\;\geqslant\;}                   

\newcommand{\norm}[1]{\left\lVert#1\right\rVert}

\title[Regularization by noise for PAM and related models]{Regularization by random translation of potentials for the continuous PAM and related models in arbitrary dimension }

\author{Florian Bechtold}
\date{\today}

\address{Florian Bechtold: Fakult\"at f\"ur Mathematik,
	Universit\"at Bielefeld, 33501 Bielefeld, Germany}
\email{fbechtold@math.uni-bielefeld.de}

\begin{document}

\maketitle
\begin{abstract}
  We study a regularization by noise phenomenon for the continuous parabolic Anderson model with a potential shifted along paths of fractional Brownian motion. We demonstrate that provided the Hurst parameter is chosen sufficiently small, this shift allows to establish well-posedness and stability to the corresponding problem - without the need of renormalization - in any dimension. We moreover provide a robustified Feynman-Kac type formula for the unique solution to the regularized problem building upon regularity estimates for the local time of fractional Brownian motion as well as non-linear Young integration.
\end{abstract}

\section{Introduction}
Consider the problem
\[
\partial_t u=\frac{1}{2}\Delta u-Vu,\qquad u(0)=f.
\]
As it is well known, provided $f$ and $V$ are sufficiently smooth, we obtain the unique solution to the above problem via the Feynman-Kac formula
\[
u(t,x)=\mathbb{E}^x\left[f(W_t)\exp{\left(-\int_0^t V(s, W_{t-s})ds\right)}\right],
\]
where $(W_t)_t$ is a standard Brownian motion in $\R^d$ on a stochastic basis $(\Omega, \mathcal{F}, \mathbb{P}, (\mathcal{F}_t)_t)$ and $\mathbb{E}^x[(\cdot)]$ denotes the expectation conditional on the Brownian motion starting in $x\in \R^d$. In the case of $V$ enjoying only distributional regularity, this reasoning is no longer applicable. A famous example in this context is the continuous parabolic Anderson model (PAM)
\begin{equation}
    \partial_t u=\frac{1}{2}\Delta u-\xi u,\qquad u(0)=f,
    \label{pam}
\end{equation}
where $\xi$ is spatial white noise on $\mathbb{T}^d$ or $\R^d$.  While in the case $d=1$ \eqref{pam} is still well posed due to the regularizing effect of the Laplacian, already $d=2,3$ require renormalization and the implementation of advanced tools from the theory of singular SPDEs such as regularity structures or paracontrolled distributions \cite{LABBE20193187}, \cite{hairerlabbe}, \cite{hairerlabbe1}, \cite{gubinelli_imkeller_perkowski_2015}, \cite{allez2015continuous}, \cite[Example 1.21]{koenig}. These approaches break down for $d\geq 4$. \\
\\
In the following, we intend to study evolution problems of the above form with a potential shifted along paths of fractional Brownian motion $w^H$ of Hurst parameter $H$, i.e.
\begin{equation}
   \partial_t u=\frac{1}{2}\Delta u-\tilde{V} u,\qquad u(0)=f,
   \label{regularized problem}
\end{equation}
where  we denote $\tilde{V}(t, x):=V(x-w^H_t)$ in case $V$ enjoys sufficient regularity to admit pointwise evaluations or respectivelty $\tilde{V}(t, \varphi)=\langle V, \varphi(\cdot+w^H_t)\rangle$ for any smooth test function $\varphi$ in case $V$ only enjoys some distributional regularity. We show that by shifting the reference frame of the potential in such a way, we can restore well-posedness to \eqref{regularized problem} and obtain stability results even in the setting of distributional valued $V$. As an application, we establish well-posedness of a shifted parabolic Anderson model
\begin{equation}
   \partial_t u=\frac{1}{2}\Delta u-\tilde{\xi} u,\qquad u(0)=f
   \label{regularized PAM}
\end{equation}
in arbitrary dimension, provided the Hurst parameter $H$ is chosen sufficiently small. \\
\\
Towards this end, we exploit a pathwise regularization phenomenon on the level of the Feynman-Kac formula building mainly on  \cite{harang2020cinfinity}. More concretely, we establish in Lemma \ref{identification} below that for smooth potentials $V^\epsilon$, we may rewrite the Feynman-Kac formula for \eqref{regularized problem} as 
\begin{equation}
    u^\epsilon(t,x)=\mathbb{E}^x\left[f(W_t)\exp{\left(-(\mathcal{I}A^{t, \epsilon})_{t}\right)}\right],
    \label{feynman kac regularized}
\end{equation}
where $\mathcal{I}$ denotes Gubinelli's Sewing operator\footnote{Refer to  \cite{gubi} and Lemma \ref{sewing} in the Appendix.}, $A^{t,\epsilon}$ the germ
\[
A^{t, \epsilon}_{s, r}=(V^\epsilon*L_{s,r})(W_{t-s}),
\]
and $L$ the local time associated with the fractional Brownian motion $w^H$. Due to Young's inequality in Besov spaces, \eqref{feynman kac regularized} might be well defined even for distributional potentials $V$, provided the local time $L$ enjoys sufficient spatial regularity\footnote{This is precisely what \cite[Theorem 17]{harang2020cinfinity}, also cited explicitly below, provides us with.}. We therefore refer to \eqref{feynman kac regularized} as a robustified Feynman-Kac formula. This main observation allows us to subsequently pass by a mollification argument: Given a distributional $V$, we consider first a mollification $V^\epsilon=\rho^\epsilon*V$. For such $V^\epsilon$, we may express the unique solution $u^\epsilon$ to the associated PDE as in \eqref{feynman kac regularized} by Lemma \ref{identification}. Next, in Lemma \ref{convergence of mollifications} we establish that due to the regularizing effect of the local time $L$, $(u^\epsilon)_\epsilon$ will converge in appropriate topologies to
\begin{equation}
\label{limit candidate}
    u(t,x)=\mathbb{E}^x\left[f(W_t)\exp{\left(-(\mathcal{I}A^{t})_{t}\right)}\right],
\end{equation}
that then becomes a candidate for a solution to \eqref{regularized problem}. Notice that the expression $\tilde{V}u$ appearing in the equation is however a priori ill-defined. Yet, exploiting our explicit robustified Feynman-Kac representation \eqref{limit candidate}, we can easily establish higher spatial regularity of $u$. In particular, we demonstrate that provided the initial condition $f$ is chosen sufficiently smooth and $H$ sufficiently small, $u$ enjoys sufficient regularity for the product 
\[
V(\cdot)u(t, \cdot+w^H_t)
\]
to be well defined in the sense of Lemma \ref{multiplication lemma} for any $t\in [0,T]$. Under these more restrictive conditions, we are thus able to identify \eqref{limit candidate} as a weak solution to \eqref{regularized problem}.

\subsection{Formulation of the main result}
\begin{theorem}
\label{main theorem}
For $\eta\geq 0$ and $d\in \mathbb{N}$, let $V\in H^{-\eta}(\R^d)$ and $f\in C^1(\R^d)$. Let $w^H$ be a $d$-dimensional fractional Brownian motion on $(\Omega^H, \mathcal{F}^H, \mathbb{P}^H)$ whose Hurst parameter satisfies $H<\frac{1}{2}(1+\eta+d/2)^{-1}$. Then for any mollification $V^\epsilon=V*\rho^\epsilon$ and $\tilde{V^\epsilon}(t,x):=V^\epsilon(x-w^H_t)$, the sequence $(u^\epsilon)_\epsilon$ of unique solutions to the the problem 
\[
\partial_t u^\epsilon=\frac{1}{2}\Delta u^\epsilon-\tilde{V^\epsilon}u^\epsilon, \qquad u(0)=f
\]
is Cauchy in $C([0,T]\times \R^d)$ equipped with the topology of uniform convergence, $\mathbb{P}^H$-almost surely. Moreover, provided further $\eta\not \in \mathbb{N}$ and $H<\frac{1}{2}(1+\eta+\lceil \eta \rceil+d/2)^{-1}$ and $f\in C^{\lceil \eta \rceil}$, we have for the limit $u$ that $u(t, \cdot)\in C^{\lceil \eta \rceil}$ uniformly in $[0,T]$,  $\mathbb{P}^H$-almost surely. In particular the product $V(\cdot)u(t, \cdot+w^H_t)$ is well defined for any $t\in [0,T]$ in the sense of Lemma \ref{multiplication lemma} and $u$ is a weak solution to 
\begin{equation}
    \partial_t u=\frac{1}{2}\Delta u-\tilde{V}u, \qquad u(0)=f
    \label{shifted problem}
\end{equation}
i.e. for any $\varphi\in C_c^\infty(\R^d)$ and $t\in [0,T]$ we have
\[
\langle u_t-f, \varphi\rangle= \int_0^t \langle  u_s, \frac{1}{2}\Delta \varphi\rangle ds+\int_0^t \langle V(\cdot) u(\cdot+w^H_s), \varphi(\cdot+w^H_s)\rangle ds.
\]
\end{theorem}
\begin{cor}[Regularized PAM]
Consider spatial white noise on the $d$-dimensional torus $\mathbb{T}^d$. Its realizations are known to lie in $H^{-(d/2+\epsilon)}$ for any $\epsilon>0$ almost surely \cite{regularity_white_noise_torus}. Hence, imposing $H<\frac{1}{2}(1+d)^{-1}$, we can apply the first part of our main result Theorem \ref{main theorem}. Demanding further $H<\frac{1}{2}(1+d+\lceil d/2+1/4 \rceil)^{-1}$, we may employ the second part, yielding a weak solution. 
\end{cor}
\begin{rem}
Note that due to the robustness of our approach, several canonical extensions to the above statements follow readily mutatis mutandis: Instead of considering the Laplacian, more general non-degenerate diffusion operators can be considered in \eqref{shifted problem}. Moreover, instead of shifting the potential $V$ along paths of fractional Brownian motion, shifts along any path that admits a sufficiently regular local time are conceivable. 
\end{rem}
\subsection{Short overview of existing literature}
The idea of employing a robustified Feynman-Kac formula in the study of heat equations with some form of multiplicative noise can be traced back to at least \cite{FKconjecture}, \cite{FK1}, \cite{FK2} for various types of space-time fractional Brownian motions.  \cite{hu_le_nonlinear} combines these considerations with non-linear Young theory similar in spirit to the setting presented in this article. Remark however that our qualitatively different approach of considering random translations of the potential allows us to treat considerably more singular potentials. Furthermore, robustifications of the Feynman-Kac formula have been employed in the study of rough stochastic PDEs for example in \cite[Chapter 12]{frizhairer} or  \cite{frizdiehlstannat}. Regularization by additive noise for the multiplicative stochastic heat equation was recently established by \cite{catellier2021pathwise} building upon ideas on pathwise regularization by noise in the spirit of \cite{gubicat}, \cite{galeati2020noiseless}, \cite{harang2020cinfinity}. Let us mention that this approach to regularization by noise has recently seen numerous interesting applications for example to interacting particle systems \cite{harang2020pathwise}, distribution dependent SDEs \cite{galeati2021distribution1}, \cite{galeati2021distribution2} and multiplicative SDEs \cite{galeati2020regularization}, \cite{bechtold}. 


\subsection{Notation}
We employ the standard notation necessary to formulate and apply the Sewing Lemma for which we refer to Appendix \ref{sewing concepts}. Let $ \mathcal{S}'$ denote the space of tempered distributions. For $\eta\in \R$, we denote by $H^\eta$ the inhomogeneous Bessel potential space of order $\eta$, i.e.
\[
H^\eta:=\left\{f\in \mathcal{S}'|\ \norm{f}_{H^\eta}=\norm{(1+|(\cdot)|)^\eta \hat{f}}_{L^2}<\infty \right\}.
\]
Moreover, for $\alpha>0$ and $\alpha\not\in \mathbb{N}$ we denote by $\mathcal{C}^\alpha$ the H\" older space
\[
\mathcal{C}^\alpha:=\left\{f\in \mathcal{S}'|\ \norm{f}_{\mathcal{C}^\alpha}= \norm{f}_{C^{\lfloor \alpha \rfloor}} +\sup_{x\neq y} \frac{|(D^kf)(x)-(D^kf)(y)|}{|x-y|^{\alpha-\lfloor \alpha \rfloor}}<\infty \right\},
\]
where for $n\in \mathbb{N}_0$
\begin{equation}
    \label{norm}
 \norm{f}_{C^n}=\sum_{k=0}^{n}\norm{D^kf}_\infty.
\end{equation}
We denote by $C^n$ the space of $n$-times continuously differentiable functions such that \eqref{norm} is finite.
Let us remark that the first two spaces above are related to more general Besov spaces in the sense that $H^\alpha=B^\alpha_{2,2}$ for any $\alpha\in \R$ and $\mathcal{C}^\alpha= B^\alpha_{\infty, \infty}$ for $\alpha>0$ and $\alpha\not \in \mathbb{N}$. In particular, note that by Young's inequality in Besov spaces \cite{kuhn2021convolution}, we have 
\begin{equation}
    \norm{f*g}_{\mathcal{C}^{\alpha-\eta}}\lesssim \norm{f}_{H^\alpha}\norm{g}_{H^{-\eta}}
    \label{young in besov}
\end{equation}
for any $\alpha-\eta$ such that $\alpha-\eta>0$ and $\alpha-\beta\not \in \mathbb{N}$. Let us also recall the multiplication theorem for Besov spaces (see e.g. {\cite[Corollary 2.1.35]{Martin2018Refinements}}, {\cite[Theorem 19.7]{willem_lec_notes}}) adapted to our setting:
\begin{lemma}
\label{multiplication lemma}
Let $\alpha>0$ such that $\alpha-\eta>0$. Then for any $\epsilon>0$  and $u, v\in \mathcal{S}'$ we have
\[
\norm{u\cdot V}_{H^{-\eta-\epsilon}}\lesssim \norm{u}_{\mathcal{C}^\alpha}\norm{V}_{H^{-\eta}},
\]
i.e. the multiplication operator extends to a continuous bilinear map $\cdot : \mathcal{C}^\alpha\times H^{-\eta}\to H^{-\eta-\epsilon} $
\end{lemma}

\section{Proof of Theorem \ref{main theorem}}
For the readers convenience we begin by stating the following result on the regularity of local times associated with fractional Brownian motion that we will use throughout. We refer to the Appendix \ref{local time concepts} the basic definitions of occupation measures, local times and the occupation times formula. 
\begin{lemma}[{\cite[Theorem 17]{harang2020cinfinity}} ]
\label{regularity of local time}
Let $w^H$ be a $d$-dimensional fractional Brownian motion of Hurst parameter $H<1/d$ on $(\Omega^H, \mathcal{F}^H, \mathbb{P}^H)$. Then there exists a null set $\mathcal{N}$ such that for all $\omega\in \mathcal{N}^c$, the path $w^H(\omega)$ has a local time $L(\omega)$ and for $\lambda<\frac{1}{2H}-\frac{d}{2}$ and $\gamma\in [0, 1-(\lambda +\frac{d}{2})H)$ we have
\begin{equation}
    \norm{L_{s,t}(\omega)}_{H^\lambda}\leq C_T(\omega) |t-s|^\gamma.
    \label{local time is Holder}
\end{equation}
for any $s, t\in [0,T]$, where $L_{s,t}=L_t-L_s$. 
\end{lemma}
Throughout the remainder of the paper and for $H<1/d$ satisfying the conditions demanded in the statements below, we shall fix a realization of fractional Brownian motion $w^H$ on $(\Omega^H, \mathcal{F}^H, \mathbb{P}^H)$ that admits a local time $L$ and for which Lemma \ref{regularity of local time} can be applied in the corresponding regularity regime.

\subsection{Solutions to the mollified equation converge}

Let $V^\epsilon$ be a mollification of some given $V\in \mathcal{S}'$. Then for 
\[
\tilde{V^\epsilon}(t,x):=V^\epsilon(x-w^H_t)
\]
we know that the unique solution to the problem
\begin{equation}
       \partial_t u^\epsilon=\frac{1}{2}\Delta u^\epsilon-(\tilde{V^\epsilon})u^\epsilon, \qquad u^\epsilon(0)=f,
       \label{mollified prob}
\end{equation}
is given by 
\begin{align*}
    u^\epsilon(t,x)
    &=\mathbb{E}^x\left[ f(W_t)\exp{\left(-\int_0^t V^\epsilon(W_{t-s}-w^H_s)ds \right)}\right],
\end{align*}
where $(W_t)_t$ is a standard Brownian motion in $\R^d$ on a stochastic basis $(\Omega, \mathcal{F}, \mathbb{P}, (\mathcal{F}_t)_t)$ and $\mathbb{E}^x[(\cdot)]$ denotes the expectation conditional on the Brownian motion starting in $x\in \R^d$. We first establish that we may replace the above Lebesgue integral in time by an appropriate sewing that is capable of leveraging the regularizing effect due to the highly oscillating fractional Brownian motion. Towards this end, we exploit the Sewing Lemma \ref{sewing}.

\begin{lemma}[Identification of Riemann integral as Sewing]
\label{identification}
For $\delta>0$, let $V\in H^{(1\vee d/2)+\delta}(\R^d)$. Let $W$ be a standard Brownian motion on $\R^d$ on a stochastic basis $(\Omega, \mathcal{F}, \mathbb{P}, (\mathcal{F}_t)_t)$. Then for  almost every $\omega\in \Omega$ the germ
\[
A^t_{s,r}:=(V*L_{s,r})(W_{t-s}(\omega))=\int_s^rV(W_{t-s}(\omega)-w^H_v)dv
\]
admits a Sewing $(\mathcal{I}A^t)$ on $[0,t]$ and moreover  for any $t\in [0,T]$ we have
\[
\int_0^t V(W_{t-s}-w^H_s)ds=(\mathcal{I} A^t)_{t}
\]

\end{lemma}
\begin{proof}
Remark first that by Lemma \ref{regularity of local time} we have for some $\epsilon>0$ and any $(s,t)\in \Delta_2([0,T])$:
\[
\norm{L_{s,t}}_{L^2}\leq C_T|t-s|^{1/2+\epsilon}.
\]
Moreover, by Young's inequality in Besov spaces \eqref{young in besov}, we  have that 
\[
\norm{V*L_{s,t}}_{\mathcal{C}^{1+\delta}}\lesssim \norm{V}_{H^{1+\delta}}\norm{L_{s,t}}_{L^2}
\]
Moreover, we have for almost every $\omega\in \Omega$ that $W\in C^{1/2-\epsilon/2}$. We therefore obtain for $(s,u,r,t)\in \Delta_4([0,T])$ that
\begin{align*}
    |(\delta A^t)_{s,u,r}|&=|V*L_{u,r}(W_{t-s})-V*L_{u,r}(W_{t-u})|\\
    &\lesssim \norm{V}_{H^{1+\delta}}\norm{L_{s,t}}_{L^2} |r-u|^{1/2+\epsilon}|u-s|^{1/2-\epsilon/2}.
\end{align*}
We conclude that $A^t$ does indeed admit a Sewing on $[0,t]$ and for $(s,u,t)\in \Delta_3([0,T])$ we have
\[
|A^t_{s,u}-(\mathcal{I} A^t)_{s,u}|=O(|u-s|^{1+\epsilon/2}).
\]
Next, observe that the germ
\[
\tilde{A}^t_{s,u}:=\int_s^uV(W_{t-v}-w^H_v)dv
\]
trivially admits a Sewing, as $\delta \tilde{A}^t=0$ wherefore we have $(\mathcal{I}\tilde{A}^t)=\tilde{A}^t$. Moreover note that because of $V\in H^{d/2+\delta}\hookrightarrow \mathcal{C}^\delta$ we have for $(s,u,t)\in \Delta_3$
\begin{align*}
   \left|\tilde{A}^t_{s,u}-A^t_{s,u} \right|\lesssim \int_s^u|W_{t-s}-W_{t-v}|^\delta dv\lesssim |u-s|^{1+\delta (1-\epsilon/2)}
\end{align*}
allowing to conclude
\begin{align*}
    \left| \int_s^uV(W_{t-v}-w^H_v)dv-(\mathcal{I} A^t)_{s, u}\right|\leq |\tilde{A}^t_{s,u}-A^t_{s,u}|+|A^t_{s,u}-(\mathcal{I} A^t)_{s,u}|\lesssim |u-s|^{1+\delta (1-\epsilon/2)}
\end{align*}

Hence the function
\[
s\in [0,t]\to \int_0^s V(W_{t-r}-w^H_r)dr-(\mathcal{I} A^t)_s
\]
is constant. Since it moreover starts in zero, this establishes the claim. 

\end{proof}
\begin{rem}
Remark that in the above statement, we did not exploit any regularization from the local time, but instead demanded regularity of the potential $V$. As a consequence, the only constraint on the Hurst parameter at this stage is $H<1/d$, which simply assures the existence of a local time. In the following, we will impose further restrictions on the Hurst parameter, allowing to pass to less regular potentials $V$.  
\end{rem}

By the previous Lemma \ref{identification}, we have that indeed
\begin{align*}
    u^\epsilon(t,x)&=\mathbb{E}^x\left[ f(W_t)\exp{\left(-\int_0^t V^\epsilon(W_{t-s}-w^H_s)ds \right)}\right]\\
    &=\mathbb{E}^x\left[ f(W_t)\exp{\left(-(\mathcal{I} A^{t,\epsilon})_{t} \right)}\right],
\end{align*}
where
\[
A^{t,\epsilon}_{s,u}:=(V^\epsilon*L_{s,u})(W_{t-s}).
\]
In the next Lemma we address the question: Under which condition on $V$ and $H$ is it possible to pass to a limit $\epsilon\to 0$?

\begin{lemma}[Convergence of mollifications]
For $\eta\geq 0$, $d\in \mathbb{N}$, let $f\in C^1(\R^d)$, $V\in H^{-\eta}(\R^d)$ and  $H<\frac{1}{2}(1+\eta+d/2)^{-1}$. Let $u^\epsilon$ be the unique solution to the mollified problem \eqref{mollified prob} and set
\begin{equation}
    u(t,x):=\mathbb{E}^x\left[ f(W_t)\exp{\left(-(\mathcal{I} A^t)_{t} \right)}\right]
    \label{limit u}
\end{equation}
where 
\[
A^t_{s,u}:=(V*L_{s,u})(W_{t-s}).
\]
Then $u^\epsilon$ converges uniformly to $u$ on $[0,T]\times \R^d$. 
\label{convergence of mollifications}
\end{lemma}
\begin{proof}
Let us start by establishing that $u$ is well defined under the conditions stated in the Lemma. By Lemma \ref{regularity of local time}, we have for some $\delta>0$
\[
\norm{L_{s,t}}_{H^{1+\eta+\delta}}\lesssim |t-s|^{1/2+\delta}.
\]
Note that again, by Young's inequality in Besov spaces, we have
\begin{align*}
    \norm{V*L_{s,t}}_{\mathcal{C}^{1+\delta}}\leq \norm{V}_{H^{-\eta}}\norm{L_{s,t}}_{H^{1+\eta+\delta}}
    \end{align*}
    meaning again in particular that $V*L_{s,t}$ lies in $C^1$. Similar to the previous Lemma \ref{identification}, we can conclude that indeed, $A^t$ admits a sewing, since
    \begin{align*}
    |(\delta A^t)_{s,u,r}|&=|V*L_{u,r}(W_{t-s})-V*L_{u,r}(W_{t-u})|\\
    &\lesssim \underbrace{\left(\sup_{x\neq y \in [0,T]} \frac{|W_x(\omega)-W_y(\omega)|}{|x-y|^{1/2-\delta/2}} \right)}_{=:c_\delta(\omega)}|r-u|^{1/2+\delta}|u-s|^{1/2-\delta/2}.
\end{align*}
    The above ensures that for almost every $\omega\in \Omega$ the expression $(\mathcal{I}A^t)_{0,t}$ is well defined. We further demonstrate that $(\mathcal{I}A^t)_{0,t}$ admits exponential moments with respect to the measure $\mathbb{P}$, allowing to establish well posedness of \eqref{limit u}. By the Sewing Lemma \ref{sewing}, we have for $t\in [0,T]$ the a priori bound
    \begin{equation}
    \begin{split}
    |(\mathcal{I}A^t)_{0,t}|&\leq |A^t_{0,t}|+|(\mathcal{I}A^t)_{0,t}-A^t_{0,t}|\\
    &\leq \norm{V*L_{0,t}}_{\infty}+\norm{\delta A^t}_{1+\delta/2} T^{1+\delta/2}\\
    &\lesssim 1+c_\delta(\omega).
    \label{a priori sewing}
    \end{split}
    \end{equation}
    Hence, we conclude that for some $a>0$, we have
\begin{align*}
    |\mathbb{E}^x\left[ f(W_t)\exp{(-(\mathcal{I}A)_{0,t})}\right]|\lesssim \norm{f}_{\infty}\mathbb{E}^x[\exp{(a c_\delta(\omega))}]<\infty
\end{align*}
by Lemma \ref{exponential moments}. This shows that the function $u$ in \eqref{limit u} is well defined as a function in $C^0([0,T]\times \R^d)$. Towards establishing convergence, let us first remark that similar to \eqref{a priori sewing}, we have 
\[
 |(\mathcal{I}A^{t,\epsilon})_{0,t}|\lesssim 1+c_\delta(\omega)
\]
uniformly in $\epsilon>0$. This permits the following bound
\begin{align*}
    |u^\epsilon_{t}(x)-u_{t}(x)|&\leq\left|\mathbb{E}^x\left[f(W_t)(e^{-(\mathcal{I} A^t)_{t} }-e^{-(\mathcal{I} A^{t,\epsilon})_{t} }) \right]\right|\\
   &\leq \mathbb{E}^x\left[|f(W_t)|e^{ac_\delta}|(\mathcal{I} A^t)_{t}-(\mathcal{I} A^{t,\epsilon})_{t}|) \right].
\end{align*}
Next, note that due to the linearity of the Sewing operator $\mathcal{I}$
\begin{align*}
    |(\mathcal{I}A^t)_{s,t}-(\mathcal{I}A^{t,\epsilon})_{t}|\leq \norm{A^t-A^{t,\epsilon}}_{1/2}T^{1/2}+\norm{\delta (A^t-A^{t,\epsilon})}_{1+\delta/2}T^{1+\delta/2}.
\end{align*}
We have moreover that
\begin{align*}
  (A^t-A^{t,\epsilon})_{s,r}&\leq \norm{(V-V^\epsilon)*L_{s,r}}_\infty\lesssim \norm{V-V^\epsilon}_{H^{-\eta}}\norm{L_{s,r}}_{H^\eta+\delta}\lesssim \norm{V-V^\epsilon}_{H^{-\eta}} |r-s|^{1/2},
\end{align*}
as well as similar to the above calculations
\begin{align*}
    (\delta(A^t-A^{t,\epsilon}))_{s,u,r}&=|(V-V^\epsilon)*L_{u,r}(W_{t-s})-(V-V^\epsilon)*L_{u,r}(W_{t-u})|\\
    &\lesssim c_\delta(\omega) \norm{V-V^\epsilon}_{H^{-\eta}} |r-u|^{1/2+\delta}|u-s|^{1/2-\delta/2}.
\end{align*}
Overall, this permits to conclude that 
\begin{align*}
    |u^\epsilon_{t}(x)-u_{t}(x)|\lesssim \norm{V-V^\epsilon}_{H^{-\eta}}\norm{f}_{\infty}\mathbb{E}^x[c_{\delta}\exp{(a(1+c_{\delta}))}]\lesssim \norm{V-V^\epsilon}_{H^{-\eta}}\norm{f}_{\infty}
\end{align*}
exploiting Lemma \ref{exponential moments}. Hence we have established the claim.
\end{proof}

\subsection{Weak solutions}
In the following section, we undertake to establish in what sense and under what conditions the function $u$ obtained in the previous section solves our original problem. We recall again that the main obstacle to address lies in the appearance of the product $\tilde V u$. This obstacle will be overcome by establishing higher spatial regularity of $u$ building upon our robustified Feynman-Kac representation for $u$ in Lemma \ref{convergence of mollifications} provided we demand more regularity in the initial condition $f$ as well as the local time $L$. In this way, Lemma \ref{multiplication lemma} allows us to give a meaning to the product $\tilde V u$, thereby allowing to conclude that $u$ satisfies the original problem in the weak sense of Theorem \ref{main theorem}. 
\begin{lemma}[Spatial regularity of $u$]
For $n\in \mathbb{N}$ suppose  $H<\frac{1}{2}(1+\eta+n+d/2)^{-1}$, $f\in C^n$ and $V\in H^{-\eta}$. Then for every $t\in [0, T]$, the function 
\begin{equation}
    u(t,x):=\mathbb{E}^x\left[ f(W_t)\exp{\left(-(\mathcal{I} A^t)_{0,t} \right)}\right],
\end{equation}
where
\[
A^t_{s,r}=(V*L_{s,r})(W_{t-s})
\]
lies in $C^n$. Moreover, $u^\epsilon\to u$ in $C^n$, uniformly in $t\in [0,T]$.
\end{lemma}
\begin{proof}
Fix $t\in [0,T]$. We show that the function
\[
x\to (\mathcal{I} A^t(x))_{t}
\]
where
\[
A_{s,r}^t(x)=(V*L_{s,r})(W_{t-s}+x)
\]
is $n$ times differentiable and that moreover, all derivatives up to order $n$ are uniformly bounded in space, integrable with respect to $\mathbb{P}$. Towards this end, let us note that 
\begin{equation}
    D^k_x(\mathcal{I} A^t(x))_{t}=(\mathcal{I} (D^k_xA^t)(x))_{t}.
    \label{derivative commutes}
\end{equation}
This can be established by using Lemma \ref{sewing convergence}. For the sake of conciseness, we restrict ourselves to the case $d=1$ and one derivative. Let us define 
\[
A^{t, n}_{s,r}=n\left((V*L_{s,r})(W_{t-s}+x+1/n)-(V*L_{s,r})(W_{t-s}+x)\right).
\]
Then it can be seen easily that 
\[
\norm{A^{t,n}-D_xA^t}_{1/2}\to 0
\]
uniformly in $t\in [0,T]$. Moreover, 
\begin{align*}
    \delta A^{t,n}_{s,u,r}&=n\left((V*L_{s,u})(W_{t-r}+x+1/n)-(V*L_{s,u})(W_{t-r}+x)\right)\\
    &-n\left((V*L_{s,u})(W_{t-u}+x+1/n)+(V*L_{s,u})(W_{t-u}+x)\right)\\
    &=(V*D_xL_{s,u})(W_{t-r}+x)-(V*D_xL_{s,u})(W_{t-u}+x)+O(1/n)|u-s|^{1+2\delta}\\&\lesssim c_\delta(\omega)|u-s|^{1/2+\delta}|r-u|^{1/2-\delta/2}+O(1/n)|u-s|^{1+2\delta}
\end{align*}
meaning that indeed $\sup_n\norm{\delta A^{t, n}}_{1+\delta/2}<\infty$, allowing to conclude \eqref{derivative commutes} by Lemma \ref{sewing convergence}. By the Faà di Bruno formula, we have 
\begin{align*}
    &\frac{d^n}{dx^n}\left( f(W_t+x)\exp{\left(-(\mathcal{I} A^t(x))_{t}\right)}\right)\\
    =&\exp{\left(-(\mathcal{I} A^t(x))_{t}\right)} \sum_{k=0}^n \binom{n}{k} \left(D^{n-k}_xf(W_t+x)\right)B_k(-(\mathcal{I} D_xA^t(x))_{t}, \dots -(\mathcal{I} D^k_xA^t(x))_{t}),
\end{align*}
where $B_k$ denotes the $k$-th complete Bell polynomial with the convention $B_0=1$.
Note in particular that $(\mathcal{I} D^k_xA^t(x))_{t}$ is uniformly bounded in space due to the regularity of the local time in this setting. Moreover, we have for any $k\leq n$ the a priori bound
\[
|(\mathcal{I}D^k_x A^t(x))_{t}|\lesssim (1+c_\delta(\omega)).
\]
We therefore have for any $a>0$ 
\[
|B_k(-D(\mathcal{I} A^t)_{0,t}, \dots D^k(\mathcal{I} A^t)_{0,t})|\lesssim \exp{(a c_\gamma(\omega))}
\]
wherefore 
\[
\frac{d^n}{dx^n}\left( f(B_t+x)\exp{\left(-(\mathcal{I} A^t)_{0,t}\right)}\right)
\]
is well defined, uniformly bounded in $x\in \R^d$ and integrable with respect to $\mathbb{P}$. Overall, this allows to conclude that indeed $u(t, \cdot)\in C^n$ for any $t\in [0,T]$. Finally, going through similar considerations for $u^\epsilon$ and remarking that 
\[
(\mathcal{I}D^k_x (A^t-A^{t, \epsilon}(x)))_t\to 0
\]
uniformly in $t, x\in [0,T]\times \R^d$, we infer that $u^\epsilon(t, \cdot)\to u(t, \cdot)$ in $C^n$ uniformly in $t\in [0,T]$. 
\end{proof}

As a corollary, invoking Lemma \ref{multiplication lemma}, we can now conclude the proof of Theorem \ref{main theorem} by observing the following:
\begin{cor}
Suppose  $H<\frac{1}{2}(1+\eta+\lceil \eta \rceil+d/2)^{-1}$, $f\in C^{\lceil \eta \rceil}$ and $V\in H^{-\eta}$. Then for any $t\in [0,T]$, the product 
\[
\langle (\tilde{V} u)_t, \varphi \rangle:= \langle V(\cdot) u(t, \cdot+w^H_t), \varphi (\cdot+w^H_t)\rangle
\]
is well defined in the sense of Lemma  \ref{multiplication lemma}. In particular, we infer that $u$ is a weak solution to the original problem. 
\end{cor}

\begin{subappendices}
\section{Appendix}

\subsection{Local time and occupation times formula}
\label{local time concepts}
We recall for the reader the basic concepts of occupation measures, local times and the occupation times formula. A comprehensive review paper on these topics is \cite{horowitz}. 
\begin{definition}
Let $w:[0,T]\to \R^d$ be a measurable path. Then the occupation measure at time $t\in [0,T]$, written $\mu^w_t$ is the Borel measure on $\R^d$ defined by 
\[
\mu^w_t(A):=\lambda(\{ s\in [0,t]:\ w_s\in A\}), \quad A\in \mathcal{B}(\R^d),
\]
where $\lambda$ denotes the standard Lebesgue measure. 
\end{definition}
The occupation measure thus measures how much time the process $w$ spends in certain Borel sets. Provided for any $t\in [0,T]$, the measure is absolutely continuous with respect to the Lebesgue measure on $\R^d$, we call the corresponding Radon-Nikodym derivative local time of the process $w$:
\begin{definition}
Let $w:[0,T]\to \R^d$ be a measurable path. Assume that there exists a measurable function $L^w:[0,T]\times \R^d\to \R_+$ such that 
\[
\mu^w_t(A)=\int_A L^w_t(z)dz, 
\]
for any $A\in \mathcal{B}(\R^d)$ and  $t\in [0,T]$. Then we call $L^w$ local time of $w$. 
\end{definition}
Note that by the definition of the occupation  measure, we have for any bounded measurable function $f:\R^d\to \R$ that 
\begin{equation}
    \int_0^tf(w_s)ds=\int_{\R^d} f(z)\mu^w_t(dz).
    \label{occupation times formula}
\end{equation}
The above equation \eqref{occupation times formula} is called occupation times formula. Remark that in particular, provided $w$ admits a local time, we also have for any $x\in \R^d$
\begin{equation}
    \int_0^tf(x-w_s)ds=\int_{\R^d} f(x-z)\mu^w_t(dz)=\int_{\R^d}f(x-z)L^w_t(z)dz=(f*L^w_t)(x).
\end{equation}
\subsection{The Sewing Lemma}
\label{sewing concepts}
We recall the Sewing Lemma due to \cite{gubi} (see also \cite[Lemma 4.2]{frizhairer}). Let $E$ be a Banach space, $[0,T]$ a given interval. Let $\Delta_n$ denote the $n$-th simplex of $[0,T]$, i.e. $\Delta_n:=\{(t_1, \dots, t_n)| 0\leq t_1\dots\leq t_n\leq T \} $. For a function $A:\Delta_2\to E$ define the mapping $\delta A: \Delta_3\to E$ via
\[
(\delta A)_{s,u,t}:=A_{s,t}-A_{s,u}-A_{u,t}.
\]
Provided $A_{t,t}=0$ we say that for $\alpha, \beta>0$ we have $A\in C^{\alpha, \beta}_2(E)$ if $\norm{A}_{\alpha, \beta}<\infty$, where
\[
\norm{A}_\alpha:=\sup_{(s,t)\in \Delta_2}\frac{\norm{A_{s,t}}_E}{|t-s|^\alpha}, \qquad \norm{\delta A}_{\beta}:=\sup_{(s,u,t)\in \Delta_3}\frac{\norm{(\delta A)_{s,u,t}}_E}{|t-s|^\beta}
\]
and $\norm{A}_{\alpha, \beta}:=\norm{A}_\alpha+\norm{\delta A}_\beta$.  For a function $f:[0,T]\to E$, we denote $f_{s,t}:=f_t-f_s$. Moreover, if for any sequence $(\mathcal{P}^n([s,t]))_n$ of partitions of $[s,t]$ whose mesh size goes to zero, the quantity 
\[
\lim_{n\to \infty}\sum_{[u,v]\in \mathcal{P}^n([s,t])}A_{u,v}
\]
converges to the same limit, we note
\[
(\mathcal{I} A)_{s,t}:=\lim_{n\to \infty}\sum_{[u,v]\in \mathcal{P}^n([s,t])}A_{u,v}.
\]

\begin{lemma}[Sewing, \cite{gubi}]
Let $0<\alpha\leq 1<\beta$. Then for any $A\in C^{\alpha, \beta}_2(E)$, $(\mathcal{I} A)$ is well defined (we say that $A$ admits the sewing $(\mathcal{I} A)$). Moreover, denoting $(\mathcal{I} A)_t:=(\mathcal{I} A)_{0,t}$, we have $(\mathcal{I} A)\in C^\alpha([0,T], E)$ and $(\mathcal{I} A)_0=0$ and for some constant $c>0$ depending only on $\beta$ we have
\[
\norm{(\mathcal{I} A)_{t}-(\mathcal{I} A)_{s}-A_{s,t}}_{E}\leq c\norm{\delta A}_\beta |t-s|^\beta.
\]
\label{sewing}
\end{lemma}

\begin{lemma}[Lemma A.2 \cite{Galeati2021}]
\label{sewing convergence}
For $0<\alpha\leq 1<\beta$ and $E$ a Banach space, let  $A\in C^{\alpha, \beta}_2(E)$ and $ (A^n)_n\subset C^{\alpha, \beta}_2(E)$ such that for some $R>0$ $\sup_{n\in \mathbb{N}}\norm{\delta A^n}_\beta\leq R$ and such that $\norm{A^n-A}_\alpha\to 0$. Then \[\norm{\mathcal{I}(A-A^n)}_\alpha\to 0.\]
\end{lemma}
\subsection{Exponential moments for the Hölder modulus of continuity of Brownian motion}
For the sake of completeness, we provide a sketch of the proof that the $\gamma$-H\" older modulus of continuity of Brownian motion is exponentially integrable for $\gamma\in (0, 1/2)$. Refer also to \cite{hoelderexponential} for more refined integrability statements. 
\begin{lemma}
Let $B$ be a standard Brownian motion. Then for any $\gamma<1/2$ and $a>0$, we have
\[
\mathbb{E}\left[\exp{\left(a\sup_{s\neq t\in [0,T]} \frac{|B_t-B_s|}{|t-s|^\gamma}\right)}\right]<\infty.
\]
\label{exponential moments}
\end{lemma}
\begin{proof}
Without loss of generality, set $T=1$. Remark first that we have for $k\in \mathbb{N}$
\[
\mathbb{E}[|B_t-B_s|^k]\leq |t-s|^{k/2} (k-1)!!
\]
We follow the classical proof of Kolmogorov's continuity theorem (refer for example to \cite[Theorem 10.1]{SchillingPartzsch+2012}). For $D_m:=2^{-m}\mathbb{N}_0\cap [0,1)$ and $D=\cup_{m}D_m$ set
\[
\Delta_m=\{ (s,t)\in D_m\times D_m :\ |t-s|\leq 2^{-m}\}.
\]
We then have for $\sigma_j:= \sup_{(s,t)\in \Delta_j} |B(t)-B(s)|$ the bound
\[
\mathbb{E}[\sigma_j^k]\leq \sum_{(s,t)\in \Delta_j}\mathbb{E}[|B_t-B_s|^k]\leq  2\cdot 2^{j(1-k/2)})(k-1)!!
\]
Following further the proof of Kolmogorov's continuity theorem as in \cite[Theorem 10.1]{SchillingPartzsch+2012}, we obtain
\begin{align*}
   \mathbb{E}\left[\left(\sup_{s\neq t\in D} \frac{|B_t-B_s|}{|t-s|^\gamma}\right)^k\right]^{1/k}&\leq 2^{1+\gamma} \sum_{j=0}^\infty 2^{j\gamma}\mathbb{E}[\sigma_j^k]^{1/k}\\
   &\leq 2^{1/k}\cdot 2^{1+\gamma} ((k-1)!!)^{1/k}\sum_{j=0}^\infty 2^{j\gamma}(2^{-j(1/2-1/k)}).
\end{align*}
Now let $k_0$ be the smallest natural such that $\gamma<1/2-1/k_0$. We then obtain for any $k\geq k_0$
\begin{align*}
    \mathbb{E}\left[\left(\sup_{s\neq t\in D} \frac{|B_t-B_s|}{|t-s|^\gamma}\right)^k\right]^{1/k}&\leq 2^{1/k}\cdot 2^{1+\gamma} ((k-1)!!)^{1/k}\sum_{j=0}^\infty 2^{j(\gamma+1/k-1/2)}\\
    &\leq 2^{1/k}\cdot 2^{1+\gamma} ((k-1)!!)^{1/k} \sum_{j=0}^\infty 2^{j(\gamma+1/k_0-1/2)}\\
    &\leq C 2^{1/k}\cdot 2^{1+\gamma} ((k-1)!!)^{1/k}.
\end{align*}
Therefore, we have 
\begin{align*}
    \sum_{k=k_0}^\infty \frac{a^k}{k!}\mathbb{E}\left[\left(\sup_{s\neq t\in [0,1]} \frac{|B_t-B_s|}{|t-s|^\gamma}\right)^k\right]&\leq 2\sum_{k=k_0}^\infty \frac{( 2^{1+\gamma}C a)^k}{k!!}<\infty,
\end{align*}
which proves the claim.
\end{proof}
\section*{Acknowledgement} The author wishes to thank Cyril Labb\'e for early discussions and Martina Hofmanová, Jörn Wichmann and Emanuela Gussetti for comments and suggestions on a preliminary version of this work. This project has received funding from the European Research Council under the European Union’s Horizon 2020 research
and innovation programme (grant agreement No.754362 and No.949981).
	\begin{figure}[ht]
	    \centering
	  \includegraphics[height=10mm]{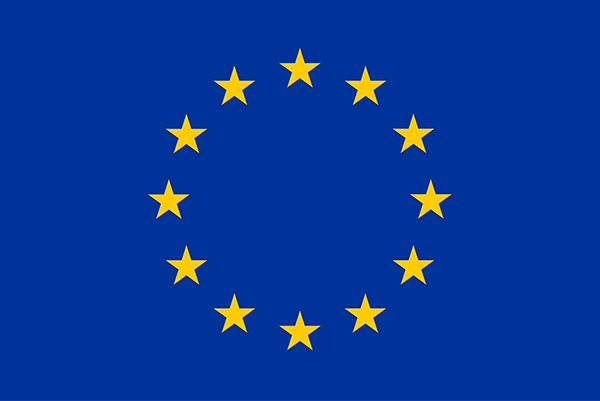}
	\end{figure}

\end{subappendices}

\bibliographystyle{alpha}
\bibliography{main}

\end{document}